\documentclass{amsart}
\usepackage{amssymb,amsmath,amsthm,hyperref}

\newtheorem{theorem}{Theorem}

\newtheorem{lemma}[theorem]{Lemma}

\newtheorem{remark}[theorem]{Remark}
\newtheorem{example}[theorem]{Example}

\begin{document}
\title{The shortest harmonic sums with decreasing denominator}

\author{Wouter van Doorn}
\address{Groningen, the Netherlands}
\email{wonterman1@hotmail.com}

\date{}

\begin{abstract}
For a positive integer $a$, let $b(a)$ be the smallest integer $b > a$ such that the denominator of $\frac{1}{a} + \frac{1}{a+1} + \cdots + \frac{1}{b}$ is smaller than the denominator of $\frac{1}{a} + \frac{1}{a+1} + \cdots + \frac{1}{b-1}$. Recently it was shown that the limit inferior $$\liminf_{a \to \infty} \left(\frac{b(a) - a}{\log a}\right)$$ exists and is positive. Here we find its exact value.
\end{abstract}

\maketitle

\section{Introduction}
For positive integers $a < b$, let $u_{a,b}$ and $v_{a,b}$ be the coprime positive integers for which $$\sum_{i=a}^b \frac{1}{i} = \frac{u_{a,b}}{v_{a,b}}.$$ For a given $a \in \mathbb{N}$, let $b(a)$ be the smallest integer $b > a$ for which $v_{a,b} < v_{a,b-1}$. Whether such a $b(a)$ exists for all $a$ was asked by Erd\H{o}s and Graham in~\cite{cnt}, recorded as Erd\H{o}s Problem \#290 at Bloom's website~\cite{EP290}, and resolved by the author in~\cite{VD}, with the upper bound $b(a) \le 4.374(a-1)$ for all $a \ge 6$. \\

As for lower bounds on $b(a)$, in \cite{VD} it was also proved that $b(a) - a \gg \log a$ holds for all $a$, but that there are infinitely many $a$ with $b(a) - a \ll \log a$. To be a bit more precise, for a positive integer $d$, let us define the polynomial $$f_d(x) := \sum_{i=0}^d \prod_{\substack{j=0 \\ j\neq i}}^d (x-j),$$ and let $\delta(f_d)$ be the density of primes $p$ such that $f_d(x) \equiv 0 \pmod{p}$ is solvable. This density exists by Chebotarev's density theorem, and the values of $\delta(f_d)$ depend on specifics of the Galois group of $f_d(x)$. For more information on this we refer to~\cite[Section 3.3]{VD}. With $$c := \sum_{d=1}^{\infty} \frac{\delta(f_d)}{d(d+1)},$$ there it was also proved that we have the string of inequalities $$0.54 < \frac{1}{1+c} \le \liminf_{a \to \infty} \left(\frac{b(a) - a}{\log a}\right) \le \frac{1}{2c} < 0.61.$$ This paper is dedicated to showing that the lower bound is sharp. 

\begin{theorem} \label{main}
We have $$\liminf_{a \to \infty} \left(\frac{b(a) - a}{\log a}\right) = \frac{1}{1+c}.$$
\end{theorem}

We will actually prove the slightly stronger statement that for all $C < 1+c$ there exists an $N$ such that for all $n \ge N$ there exist integers $a, b > e^{Cn}$ with $b = a + n$ for which $v_{a,b} < v_{a,b-1}$. We further remark that $\frac{1}{1+c}$ is approximately equal to $0.546$, although we make no effort here to calculate it as precisely as possible. Even so, if we combine Theorem~\ref{main} with~\cite[Lemma 32]{VD}, then we find that there are infinitely many $a$ and $b$ with $$a < b < a + 0.55 \log a \qquad \text{and} \qquad v_{a,b} < v_{a,b-1}.$$

\section{Reduction of the main result}
In this section we reduce the proof of Theorem~\ref{main} to the existence of a certain integer $x$. This reduction already follows from the arguments in~\cite[Section 3.3]{VD}, but in order to make this paper self-contained, we give the proof here in full. To set this all up, we need a couple of definitions. \\

Let $D \ge 2$ be an integer and let $n$ be an integer that is sufficiently large in terms of $D$. For a positive integer $d \le 2D$, let $S_d$ be the set of primes $p$ with $\frac{n}{d+1} < p \le \frac{n}{d}$ such that $f_d(x) \equiv 0 \pmod{p}$ is solvable. 

\begin{example} \label{f2pnt}
By expanding out the definition of $f_d(x)$ for $d = 2$, we find $f_2(x) = 3x^2 - 6x + 2$. For this polynomial, one can check that an odd prime $q \in \left(\frac{n}{3}, \frac{n}{2}\right]$ belongs to $S_2$ precisely when $q \equiv \pm 1 \pmod{12}$, in which case $f_2(x)$ actually has two zeros modulo $q$. Moreover, since $\frac{n}{d} - \frac{n}{d+1} = \frac{n}{d(d+1)}$, by the prime number theorem in arithmetic progressions we find $$\lvert S_2 \rvert = \left(\frac{1}{12} + o(1) \right) \frac{n}{\log n},$$ when $n$ goes to infinity.
\end{example}

We now define $Q$ as the product of all primes $q \in \bigcup_{d \le 2D} S_d$ and, for a prime factor $q$ of $Q$, set $Q_q := \frac{Q}{q}$. \\

Similarly, for a positive integer $d \le D$, we let $T_d$ be the set of primes $p$ with $\frac{n}{d+1} < p \le \frac{n}{d}$ such that $f_d(x) \equiv 0 \pmod{p}$ is \emph{not} solvable, define $P$ as the product of all primes $p \in \bigcup_{d \le D} T_d$ and set $P_p := \frac{P}{p}$ for a prime $p \mid P$. We then claim that it is sufficient to find an $x$ with $\frac{Q}{n} \le x < Q$ satisfying various congruences. 

\begin{theorem} \label{reductio}
Assume that for all integers $D \ge 2$ and all sufficiently large integers $n$ there exists an integer $x < Q$ with the following three properties:

\begin{enumerate}
	\item We have $x > \frac{Q}{n}$.
	\item For all $d \le 2D$ and all $q \in S_d$ we have $f_d(xPQ_q) \equiv 0 \pmod{q}$.
	\item For all $d \le D$ and all $p \in T_d$ we have $f_{d-1}(xP_pQ - 1) \not \equiv 0 \pmod{p}$.
\end{enumerate}

Then $$\liminf_{a \to \infty} \left(\frac{b(a) - a}{\log a}\right) \le \frac{1}{1+c}.$$
\end{theorem}

\begin{proof}
Let $D \ge 2$ be a fixed integer, let $n$ be sufficiently large in terms of $D$, and let $x$ be as in the theorem statement. We then define $b := xPQ$ and $a := b - n$, and need to show that we have $$b \le a + \left(\frac{1}{1+c} + o(1) \right) \log a \qquad \text{and} \qquad v_{a,b} < v_{a,b-1}.$$ As for the first inequality, we apply the fact that the densities $\delta(f_d)$ exist for fixed $d$, which by the prime number theorem give
\begin{align*}
P &= \exp\left[\left(\sum_{d=1}^{D} \frac{1 - \delta(f_d)}{d(d+1)} + o(1)\right)n\right],\\
Q &= \exp\left[\left(\sum_{d=1}^{2D} \frac{\delta(f_d)}{d(d+1)} + o(1)\right)n\right],\\
x &= Q^{1 + o(1)} = \exp\left[\left(\sum_{d=1}^{2D} \frac{\delta(f_d)}{d(d+1)} + o(1)\right)n\right].
\end{align*}
Here, the $o(1)$ terms go to $0$ as $n$ goes to infinity. By then letting $D$ go to infinity as well, we find $$b = xPQ = \exp\left[\big(1 + c + o(1)\big)n\right],$$ from which the desired inequality follows. \\

It remains to show that the above definitions of $a$ and $b$ imply $v_{a,b} < v_{a,b-1}$. For this, we make use of the $p$-adic valuation $\nu_p$. Since $x < Q$ by assumption, in order to show $v_{a,b} < v_{a,b-1}$ it suffices to show that the two inequalities $$\nu_q(v_{a,b}) < \nu_q(v_{a,b-1}) \qquad \text{and} \qquad \nu_p(v_{a,b}) \le \nu_p(v_{a,b-1}) + \nu_p(x)$$ hold for all $q \mid Q$ and $p \nmid Q$ respectively. Indeed, when multiplying over all primes, these two sets of inequalities imply $\frac{v_{a,b}}{v_{a,b-1}} \le \frac{x}{Q} < 1$, finishing the proof. Hence, let us first show that $\nu_q(v_{a,b}) < \nu_q(v_{a,b-1})$ holds for all $q \mid Q$. \\

Let $q$ be an arbitrary prime factor of $Q$, assume without loss of generality $q \in S_d$, and recall the definition of $f_d(x)$. For an integer $i \in \{0, 1, \ldots, d\}$ we see that all terms in $f_d(i)$ vanish, except for the $i$th term. In particular, as $$q > \frac{n}{2D+1} > 2D \ge d$$ if $n$ is sufficiently large, we find $$f_d(i) = \prod_{\substack{j=0 \\ j \neq i}}^d (i-j) \not\equiv 0 \pmod{q}.$$ We therefore conclude from the second property that $xPQ_q \not \equiv i \pmod{q}$ for all $i \in \{0, 1, \ldots, d\}$. This implies that, among the multiples $b, b-q, \ldots, b - dq$ of $q$ in the interval $[a, b]$, none of these multiples are divisible by $q^2$, since $$\frac{b - iq}{q} \equiv xPQ_q - i \not \equiv 0 \pmod{q}.$$ By multiplying the partial harmonic sum by $q$, all terms vanish modulo $q$ except for the terms with denominator equal to one of $b, b - q, \ldots, b - dq$. Hence, $$q \sum_{i = a}^b \frac{1}{i} \equiv \sum_{i = 0}^d \frac{1}{xPQ_q - i} \equiv \frac{f_d(xPQ_q)}{\prod_{i=0}^d (xPQ_q - i)} \equiv 0 \pmod{q},$$ so that $\nu_q(v_{a,b}) = 0$. On the other hand, omitting the final term $\frac{1}{b}$ in the previous calculation gives $$q \sum_{i = a}^{b-1} \frac{1}{i} \equiv -\frac{1}{xPQ_q} \not \equiv 0 \pmod{q},$$ which implies $\nu_q(v_{a,b-1}) \ge 1$, as desired. In fact, $\nu_q(v_{a,b-1})$ is exactly equal to $1$, since we just saw that none of the multiples of $q$ in $[a, b]$ are divisible by $q^2$. \\

To finish the proof, we need to show that $$\nu_p(v_{a,b}) \le \nu_p(v_{a,b-1}) + \nu_p(x)$$ holds for every prime $p \nmid Q$. This is not hard to see for $p \nmid P$, as we then get $\nu_p(b) = \nu_p(x)$, while the inequalities $$\nu_p(v_{a,b}) \le \max\big(\nu_p(v_{a,b-1}), \nu_p(b)\big) \le \nu_p(v_{a,b-1}) + \nu_p(b)$$ hold in general. Hence, for the rest of the proof we may assume $p \mid P$, so let $p \in T_d$, say. \\

As before, the multiples of $p$ in the interval $[a, b]$ are exactly $b, b - p, \ldots, b - dp$. If $xP_pQ \not \equiv i \pmod{p}$ for all $0 \le i \le d$, then none of those multiples of $p$ are divisible by $p^2$, so that the third property gives $$p \sum_{i=a}^{b-1} \frac{1}{i} \equiv \sum_{i=1}^d \frac{1}{xP_pQ - i} \equiv \frac{f_{d-1}(xP_pQ - 1)}{\prod_{i=1}^d (xP_pQ - i)} \not \equiv 0 \pmod{p}.$$ In particular, $\nu_p(v_{a,b-1}) = 1 \ge \nu_p(v_{a,b})$, which suffices. On the other hand, if $xP_pQ$ is congruent to $i \pmod{p}$ for some $0 \le i \le d$, then $b - ip$ is divisible by $p^2$. Moreover, since $p > \frac{n}{d+1}$ implies $p > d$ as long as $n$ is large enough, $b - ip$ is then the unique integer in the interval $[a, b]$ divisible by $p^2$. In this case, $i \neq 0$ implies $$\nu_p(v_{a,b}) = \nu_p(b - ip) = \nu_p(v_{a,b-1}) \le \nu_p(v_{a,b-1}) + \nu_p(x),$$ while $i = 0$ implies $\nu_p(v_{a,b-1}) = 1$ by the third property again, so that 
\begin{equation*}
\nu_p(v_{a,b}) = \nu_p(b) = \nu_p(P) + \nu_p(x) = 1 + \nu_p(x) = \nu_p(v_{a,b-1}) + \nu_p(x). \qedhere
\end{equation*}
\end{proof}

\section{Proof of the main result}
In this section we prove Theorem~\ref{main} by constructing an $x$ for which the three properties in Theorem~\ref{reductio} hold. As the lower bound on the liminf of $\frac{b(a) - a}{\log a}$ is precisely~\cite[Lemma 31]{VD}, this is sufficient.

\begin{proof}[Proof of Theorem~\ref{main}]
Let $D \ge 2$ be given, let $n$ be sufficiently large in terms of $D$, and let the products $P$ and $Q$ be as before. We define the following zeros of $f_d$ (or a transformed version thereof) for prime factors of $P$ and $Q$:

\begin{itemize}
	\item For every $q \in S_2$, let $x_q^{(1)}$ and $x_q^{(2)}$ denote the two distinct zeros of $f_2(x) \pmod{q}$.
	\item Let $x_{q_0}^{(1)}, \ldots, x_{q_0}^{(2D)}$ denote the $2D$ distinct zeros of $f_{2D}(x) \pmod{q_0}$ for some specific $q_0 \in \left(\frac{n}{2D+1}, \frac{n}{2D} \right]$ for which $f_{2D}$ splits completely modulo $q_0$. The existence of such a $q_0$ is guaranteed by Chebotarev's density theorem, assuming $n$ is large enough.
		\item For every $d \le 2D$ different from $2$ and every $q \in S_d$ different from $q_0$, let $x_q^{(1)}$ be any zero of $f_d(x) \pmod{q}$.
		\item For every $d \le D$ and every $p \in T_d$, let $x_p^{(1)}, \ldots, x_p^{(e_p)}$ denote the distinct values in $[1, p]$ for which $f_{d-1}(x_p^{(i)}P_pQ - 1) \equiv 0 \pmod{p}$ for all $i$.
\end{itemize}

\begin{remark} \label{epsmall}
For all $d \le D$ and all $p \in T_d$ we have $e_p < D$. Indeed, $f_{d-1}$ has degree $d-1$, so if $n$ is large enough, then $f_{d-1} \pmod{p}$ has at most $d-1 < D$ roots for any $p \in T_d$.
\end{remark}

For $1 \le k \le 2D$, let $x_0^{(k)}$ be the smallest positive integer with, for all $q \mid Q$ different from $q_0$, $$x_0^{(k)} \equiv x_q^{(1)}(PQ_q)^{-1} \pmod{q} \qquad \text{and} \qquad x_0^{(k)} \equiv x_{q_0}^{(k)}(PQ_{q_0})^{-1} \pmod{q_0}.$$ These $x_0^{(1)}, \ldots, x_0^{(2D)}$ are our initial potential values for $x$ and they ensure the second property. Essentially all we will do from here is change the congruence $$x \equiv x_q^{(i)}(PQ_q)^{-1} \pmod{q} \qquad \text{to} \qquad x \equiv x_q^{(i')}(PQ_q)^{-1} \pmod{q}$$ for various $q \in S_2 \cup \{q_0\}$ to obtain a final $x$ for which the first and third properties hold as well. Now, for a $q \in S_2$, let $t_q \in [1, q)$ be the smallest positive integer with $$t_qQ_q \equiv \left(x_q^{(2)} - x_q^{(1)} \right) (2PQ_q)^{-1} \pmod{q}.$$ Then switching from $x_q^{(1)}$ to $x_q^{(2)}$ changes $x$ by $2t_qQ_q \pmod{Q}$. These changes for $q \in S_2$ in principle give $2^{\lvert S_2 \rvert}$ values of $x$ to work with, but we will not need to consider all of them. \\

Let $S'_2$ be equal to $S_2$ but with the largest prime removed if $\lvert S_2 \rvert$ is odd, and further define $x_1^{(k)}$ for a positive integer $k \le 2D$ to be the unique integer in $[1, Q]$ with $$x_1^{(k)} \equiv x_0^{(k)} + \sum_{q \in S'_2} t_q Q_q \pmod{Q}.$$ In a sense, $x_1^{(k)}$ is the halfway point between making all swaps for $q \in S'_2$. We now order the primes $q \in S'_2$ according to the value of $t_qQ_q$. That is, we write $S'_2 = \{q_1, q_2, \ldots, q_{\lvert S'_2 \rvert}\}$ with $t_{q_i}Q_{q_i} < t_{q_j}Q_{q_j}$ if $i < j$. We note that these values of $t_{q_i}Q_{q_i}$ are all distinct, as otherwise such a value would be divisible by both $Q_{q_i}$ and $Q_{q_j}$ for distinct $i, j$. However, this would imply that $t_{q_i}Q_{q_i}$ is divisible by $Q$, which is impossible as $1 \le t_{q_i}Q_{q_i} < Q$, by definition of $t_{q_i}$. \\ 

For every odd $i < \lvert S'_2 \rvert$ we only consider those values of $x$ we obtain by switching exactly one of $q_i$ and $q_{i+1}$ to its second root. With $$\Delta_i := t_{q_{i+1}}Q_{q_{i+1}} - t_{q_i}Q_{q_i} > 0,$$ switching $q_i$ to its second root decreases the residue by $\Delta_i$ relative to $x_1^{(k)}$, while switching $q_{i+1}$ to its second root increases it by $\Delta_i$. If we consider all these paired changes, then the final result is congruent to $$x_1^{(k)} + \sum_{i < \lvert S'_2\rvert \text{ odd}} \varepsilon_i \Delta_i \pmod{Q},$$ where the $\varepsilon_i \in \{-1, 1\}$ can be chosen arbitrarily. Moreover, $$\left \lvert \sum_{i < \lvert S'_2\rvert \text{ odd}} \varepsilon_i\Delta_i \right \rvert \le \sum_{i < \lvert S'_2\rvert \text{ odd}} \Delta_i \le \sum_{i = 1}^{\lvert S'_2\rvert - 1} (t_{q_{i+1}}Q_{q_{i+1}} - t_{q_i}Q_{q_i}) < Q,$$ due to the fact that the final sum telescopes and $1 \le t_q Q_q  < Q$ for all $q$. We now need two lemmas on the prime factors of $\Delta_i$ and $\varepsilon_i \Delta_i + \varepsilon_j \Delta_j$.

\begin{lemma} \label{atleastonesix}
For all but at most two primes $p \mid P$, at least $\frac{1}{6} \lvert S'_2\rvert$ of the $\Delta_i$ are non-zero modulo $p$.
\end{lemma}

\begin{proof}
Since $$\Delta_i = t_{q_{i+1}}Q_{q_{i+1}} - t_{q_i}Q_{q_i} = \frac{Q}{q_iq_{i+1}}(q_it_{q_{i+1}} - q_{i+1}t_{q_i}),$$ a prime $p \mid P$ divides $\Delta_i$ precisely when $p$ divides $q_i t_{q_{i+1}} - q_{i+1} t_{q_i}$. This latter integer is smaller than $n^2$, while every prime divisor of $P$ is larger than $\frac{n}{D+1}$. Hence, if $n$ is large enough, there can be at most two prime factors of $P$ that divide $q_i t_{q_{i+1}} - q_{i+1} t_{q_i}$. As there are a total of $\frac{1}{2} \lvert S'_2 \rvert$ values of $\Delta_i$ to consider, we deduce $$\sum_{p \mid P} \sum_{\substack{i < \lvert S'_2\rvert \text{ odd} \\ p \mid \Delta_i}} 1 \le \lvert S'_2 \rvert.$$ It follows that there are at most two exceptional primes $p \mid P$ for which more than $\frac{1}{3}\lvert S'_2\rvert$ of the $\Delta_i$ vanish modulo $p$. For every other $p \mid P$, at least $\frac{1}{6} \lvert S'_2\rvert$ of the $\Delta_i$ are non-zero modulo $p$.
\end{proof}

Let $p$ be a non-exceptional prime divisor of $P$. We then define $R_p$ to be the number of quadruples $(i, j, \varepsilon_i, \varepsilon_j)$ for which $p$ divides $\varepsilon_i \Delta_i + \varepsilon_j \Delta_j$, where $i, j \in \{1, 3, \ldots, \lvert S'_2 \rvert - 1 \}$ and $\varepsilon_i, \varepsilon_j \in \{-1, 1\}$.

\begin{lemma} \label{rupper}
We have $$\sum_{\substack{p \mid P \\ p \text{ non-exceptional}}} R_p \le 19 \lvert S'_2 \rvert^2.$$
\end{lemma}

\begin{proof}
We note that $\sum_{p} R_p$ is equal to the number of quintuples $(i, j, \varepsilon_i, \varepsilon_j, p)$ for which $p$ divides $\varepsilon_i \Delta_i + \varepsilon_j \Delta_j$. Now, if a prime $p$ divides $\varepsilon_i \Delta_i + \varepsilon_j \Delta_j$, then $p$ either divides both $\Delta_i$ and $\Delta_j$, or neither. By using Lemma~\ref{atleastonesix} and its proof, the number of triples $(i, j, p)$ where a non-exceptional prime divisor of $P$ divides both $\Delta_i$ and $\Delta_j$ is at most $$\sum_{\substack{p \mid P \\ p \text{ non-exceptional}}} \left(\sum_{\substack{i < \lvert S'_2\rvert \text{ odd} \\ p \mid \Delta_i}} 1 \right)^2 \le \frac{1}{3} \lvert S'_2 \rvert \sum_{p \mid P} \sum_{\substack{i < \lvert S'_2\rvert \text{ odd} \\ p \mid \Delta_i}} 1 \le \frac{1}{3} \lvert S'_2 \rvert^2.$$ Multiplying by $4$ to account for the signs of $\varepsilon_i$ and $\varepsilon_j$ gives an upper bound for the corresponding number of quintuples $(i, j, \varepsilon_i, \varepsilon_j, p)$ where $p$ divides both $\Delta_i$ and $\Delta_j$. \\

As for the case where $p$ divides neither $\Delta_i$ nor $\Delta_j$, let us first consider the tuples with $i = j$. In this instance, $p$ only divides $\varepsilon_i \Delta_i + \varepsilon_j \Delta_j$ if $\varepsilon_i = -\varepsilon_j$, so for every $p \mid P$ there are at most $\lvert S'_2 \rvert$ such tuples. As the number of prime factors of $P$ is at most the number of primes smaller than or equal to $n$, which in turn is smaller than $13 \lvert S'_2 \rvert$ by the bound in Example~\ref{f2pnt}, the number of quintuples we are interested in with $i = j$ is at most $13 \lvert S'_2 \rvert^2$. \\

Finally, if $i \neq j$ and $p$ divides $\varepsilon_i \Delta_i + \varepsilon_j \Delta_j$, then $p$ must divide $$\frac{q_iq_{i+1}q_jq_{j+1}(\varepsilon_i \Delta_i + \varepsilon_j \Delta_j)}{Q} = \varepsilon_i q_j q_{j+1} (q_i t_{q_{i+1}} - q_{i+1} t_{q_i}) + \varepsilon_j q_i q_{i+1} (q_j t_{q_{j+1}} - q_{j+1} t_{q_j}).$$ This quantity is non-zero as it is congruent to $-\varepsilon_i q_j q_{j+1} q_{i+1} t_{q_i} \not \equiv 0 \pmod{q_i}$, while its absolute value is smaller than $n^4$. Hence, it has at most four prime factors $p \mid P$, once again due to $p > \frac{n}{D+1}$. As the number of quadruples $(i, j, \varepsilon_i, \varepsilon_j)$ is at most $\lvert S'_2\rvert^2$, we conclude that, summed over all non-exceptional prime factors of $P$, 
\begin{equation*}
\sum_p R_p \le \left(\frac{4}{3} + 13 + 4\right) \lvert S'_2 \rvert^2 < 19 \lvert S'_2 \rvert^2. \qedhere
\end{equation*} 
\end{proof}

We now choose the $\varepsilon_i$ independently and uniformly from $\{-1,1\}$ and apply~\cite[Theorem 1.4]{FJLS}, which is an inequality of Hal\'asz~\cite{Hal} over $\mathbb{F}_p$. This inequality bounds the probability that a signed sum is congruent to a specific residue class in terms of the quantity $R_p$ we just considered. More precisely, in the notation of~\cite{FJLS}, we let
\begin{align*}
\mathbf{a} &:= (\Delta_1, \Delta_3, \ldots, \Delta_{\lvert S'_2\rvert - 1}), \\
k &:= 1, \\
M &:= \frac{1}{180} \lvert S'_2\rvert, \\
n &:= \frac{1}{2} \lvert S'_2 \rvert, \\
R_k(\mathbf{a}) &:= R_p.
\end{align*}
The hypotheses of~\cite[Theorem 1.4]{FJLS} are then satisfied, and we find that the probability that the sum $$\sum_{i < \lvert S'_2\rvert \text{ odd}} \varepsilon_i\Delta_i$$ is congruent to a given residue class mod $p$ is bounded by $$\frac{1}{p} + \frac{CR_p}{ \lvert S'_2 \rvert^{5/2}} + e^{- \frac{1}{180} \lvert S'_2\rvert},$$ for some absolute constant $C$. In particular, as $e_p < D$ for any prime factor $p$ of $P$ by Remark~\ref{epsmall}, we have $$\Pr\left(\exists j \le e_p : x_1^{(k)} + \sum_{i < \lvert S'_2\rvert \text{ odd}} \varepsilon_i\Delta_i \equiv x_p^{(j)} \pmod{p}\right) < D \left(\frac{1}{p} + \frac{CR_p}{\lvert S'_2 \rvert^{5/2}} + e^{- \frac{1}{180} \lvert S'_2\rvert}\right).$$ Summing over all non-exceptional $p \mid P$ and using Lemma~\ref{rupper} shows that the probability that there exists a non-exceptional prime $p \mid P$ and a $1 \le j \le e_p$ such that $$x_1^{(k)} + \sum_{i < \lvert S'_2\rvert \text{ odd}} \varepsilon_i\Delta_i \equiv x_p^{(j)} \pmod{p}$$ is smaller than
\begin{align*}
&D \sum_{\substack{p \mid P \\ p \text{ non-exceptional}}} \left(\frac{1}{p} + \frac{CR_p}{\lvert S'_2 \rvert^{5/2}} + e^{- \frac{1}{180} \lvert S'_2\rvert}\right) \\
&\le D \sum_{\frac{n}{D+1} < p \le n} \left(\frac{1}{p} + e^{- \frac{1}{180} \lvert S'_2\rvert}\right) + \frac{CD}{\lvert S'_2 \rvert^{5/2}} \sum_{\substack{p \mid P \\ p \text{ non-exceptional}}} R_p \\
&\ll \frac{D \log D}{\log n} + \frac{D}{\sqrt{\lvert S'_2 \rvert}} \\
&< \frac{1}{6D},
\end{align*}
for all sufficiently large $n$. Now, the term $x_1^{(k)} + \sum_{i < \lvert S'_2\rvert \text{ odd}} \varepsilon_i\Delta_i$ is contained in $(-Q, 2Q)$, so in order to get a potential value of $x \in [1, Q)$, we may need to add or subtract $Q$. Fortunately, the above bounding argument works just as well for $$\Pr\left(\exists j \le e_p : \pm Q + x_1^{(k)} + \sum_{i < \lvert S'_2\rvert \text{ odd}} \varepsilon_i \Delta_i \equiv x_p^{(j)} \pmod{p}\right).$$ By the union bound the probability for one of these three terms to be congruent to some $x_p^{(j)} \pmod{p}$ is then smaller than $\frac{1}{2D}$. By another application of the union bound we see that there exists a choice of $\varepsilon_i$ such that $f_{d-1}(xP_pQ - 1) \not \equiv 0 \pmod{p}$ holds for all non-exceptional $p \mid P$ and all $2D$ values of $x \in [1, Q)$ for which $$x \equiv x_1^{(k)} + \sum_{i < \lvert S'_2\rvert \text{ odd}} \varepsilon_i \Delta_i \pmod{Q}$$ for some $1 \le k \le 2D$. Let us fix this choice of $\varepsilon_i$ and denote the resulting values of $x$ by $x^{(1)}, \ldots, x^{(2D)}$. \\

As for the exceptional primes, we claim that for all distinct $k, k' \in \{1, \ldots, 2D\}$ and all $p \mid P$ (so in particular for the exceptional primes) we have $x^{(k)} \not \equiv x^{(k')} \pmod{p}$. Indeed, $x^{(k)}$ and $x^{(k')}$ differ by $t Q_{q_0}$ for some $0 < t < q_0$, while $p \nmid Q_{q_0}$ and $$0 < t < q_0 \le \frac{n}{2D} \le \frac{n}{D+1} < p.$$ Since $f_{d-1} \pmod{p}$ has at most $D-1$ zeros for the at most two exceptional primes, while we have $2D$ values of $x^{(k)}$ to work with, we conclude that there must be at least two $x^{(k)}$, $x^{(k')}$ such that $f_{d-1}(xP_pQ - 1) \not \equiv 0 \pmod{p}$ for all $p \mid P$ and $x \in \{x^{(k)}, x^{(k')}\}$. This ensures the third property. As for the first property, $x^{(k)}$ and $x^{(k')}$ differ by at least $Q_{q_0}$, so at least one of them is larger than $Q_{q_0} > \frac{Q}{n}$. Hence, there exists an $x$ satisfying all three properties, finishing the proof.
\end{proof}

\section{Declaration of AI usage}
Extending the author's construction of an $x$ satisfying the first two properties of Theorem~\ref{reductio}, ChatGPT 5.6-Sol Pro discovered how to apply~\cite[Theorem 1.4]{FJLS} to ensure that the third property holds as well. This paper is a human-written simplification of the proof that ChatGPT came up with. For the sake of transparency, the original AI-written output is still available at~\cite{GPT}.

\end{document}